\documentclass{amsart}
\usepackage[outer=1.25in, inner=1.25in, top=1.25in, bottom=1.25in]{geometry}
\usepackage{tikz-cd}
\usetikzlibrary{positioning}
\newtheorem{theorem}{Theorem}[section]
\newtheorem*{thmA}{Theorem~A}
\newtheorem*{thmB}{Theorem~B}

\newtheorem{lemma}[theorem]{Lemma}
\newtheorem{question}[theorem]{Question}
\newtheorem{conjecture}[theorem]{Conjecture}

\def\cent#1#2{\mathrm{C}_{{#1}}{{(#2)}}}

\def\irr#1{\mathrm{Irr}(#1)}
\def\cl#1{\mathrm{Cl}(#1)}
\def\V#1{\mathrm{V}(#1)}
\usepackage{hyperref,amssymb}
\begin{document}
	
	\title[Diameter three]{On the structure of the character degree graphs having diameter three}

	\author[S.~Dolfi]{Silvio Dolfi}
	\address{Silvio Dolfi, Dipartimento di Matematica e Informatica U. Dini
Universit\`a degli Studi di Firenze, viale Morgagni 67/a, 50134 Firenze, Italy}
	\email{silvio.dolfi@unifi.it}

	\author[R.~Hafezieh]{Roghayeh Hafezieh}
	\address{Roghayeh Hafezieh, Department of Mathematics, Gebze Technical University, P.O. Box 41400, Gebze, Turkey}
	\email{roghayeh@gtu.edu.tr}

	\author[P. Spiga]{Pablo Spiga}
	\address{Pablo Spiga, Dipartimento di Matematica Pura e Applicata,\newline
		University of Milano-Bicocca, Via Cozzi 55, 20126 Milano, Italy}
	\email{pablo.spiga@unimib.it}
	\subjclass[2010]{primary 20C15}
	\keywords{character degree graph, solvable groups.}
	\thanks{The first and the third authors are members of the GNSAGA INdAM and PRIN ``Group theory and its applications'' research group and kindly acknowledge their support.\\
	*Corresponding author: Pablo Spiga, pablo.spiga@unimib.it}
	\maketitle

        \begin{abstract}
          The structure of the character degree graphs $\Delta(G)$, i.e. the prime graphs on the set $\mathrm{cd}(G)$ of the irreducible character degrees
          of a finite group $G$, such that $G$ is  solvable  and $\Delta(G)$ has diameter three, remains an intriguing area of study. However, a comprehensive understanding of these structures remains elusive. In this paper, we prove some properties and provide an infinite series of examples of this class of graphs, building on the ideas in~\cite{lewis1}.
	\end{abstract}

\section{Introduction}\label{sec:intro}
Consider a finite group $G$, and let $\mathrm{Irr}(G)$ be the collection of all irreducible complex characters of $G$. We define $\mathrm{cd}(G)$ as the set containing the degrees of these characters, denoted by $\chi(1)$, where $\chi$ belongs to $\mathrm{Irr}(G)$. Consequently, the \emph{character degree graph} $\Delta(G)$ is established as a graph with its vertex set $\V G$, which comprises all prime numbers dividing some degree $\chi(1)$ in $\mathrm{cd}(G)$. In this graph, two distinct primes, $p$ and $q$, are deemed adjacent if their product $pq$ divides any degree within $\mathrm{cd}(G)$.

The exploration of the character degree graph $\Delta(G)$ and the correlations between its properties and the structural characteristics of the group $G$ constitute a well-explored area with an extensive body of literature. For a comprehensive overview of this subject, we recommend consulting the survey paper~\cite{lewis2}. The objective of this paper is to contribute to a specific facet of this research endeavor: the structure of $\Delta(G)$, when $G$ is a solvable group and  $\Delta(G)$ has diameter three.

A pivotal result established by P.~P.~Palfy~\cite{palfy1} asserts that, if $G$ is a solvable group then, for any three distinct primes in the set $\V G$, at least two of them form an adjacent pair in the character degree graph $\Delta(G)$. Consequently, it promptly follows that, for a solvable group $G$, $\Delta(G)$ can exhibit at most two connected components, which are complete if $\Delta(G)$ is not connected. Additionally, in the case of a connected $\Delta(G)$, the diameter of $\Delta(G)$ is restricted to a maximum of $3$; the validity of this inequality in any finite group is demonstrated in~\cite{lewis3}.
Moreover, Palfy proved in~\cite{palfy2} that there is a remarkable difference between the sizes of the two connected components
of a disconnected  character degree graph of a solvable group. Namely,  if $m$ and $n$ are their sizes,  say $n \geq m$, then  $n \geq 2^m -1$.

For a considerable time, it remained an open question whether solvable groups with character degree graphs of diameter $3$ existed. This uncertainty persisted until Mark Lewis resolved the matter in~\cite{lewis1}, presenting a beautiful solvable group $G$ with $\Delta(G)$ comprising $6$ vertices and possessing a diameter of $3$, see Figure~\ref{figure2}.

The structure of the solvable groups $G$ such that the character degree graph $\Delta(G)$ has diameter three, as well as some
properties of the graph $\Delta(G)$, have been further studied in~\cite{CDPS} and~\cite{sass}.
In particular, it turns out that in this case the group $G$ has a unique non-abelian Sylow subgroup $P$ and that $\Delta(G/\gamma_3(P))$,
where $\gamma_3(P) = [P', P]$ is the third term of the descending central series of $P$, is a disconnected subgraph of $\Delta(G)$,
with the same vertex set. Let $\pi_0$ and $\pi_1$ denote the connected components of $\Delta(G/\gamma_3(P))$, with the notation chosen such that the prime divisor $p$ of $|P|$ belongs to $\pi_1$. In this context, it is established that $|\pi_1| \geq 2^{|\pi_0|}$ (\cite[Remark 4.4]{CDPS} and \cite[Theorem 4]{sass}).
So, the vertex set of  $\Delta(G)$ is covered by two cliques with vertex sets $\pi_0$ and $\pi_1$ of rather different sizes.

Let $\Delta = \Delta(G)$,  for $G$ solvable, be a graph of diameter three. Proceeding as in~\cite{sass},
we denote by $\alpha_{\Delta} \subseteq \pi_0$ the set of vertices of $\pi_0$ that are not adjacent to any vertex in $\pi_1$ and
by $\delta_{\Delta} \subseteq \pi_1$ the set of vertices of $\pi_1$ that are not adjacent to any vertex in $\pi_0$.
Since $\Delta$ has diameter $3$, every pair of vertices $s \in \alpha_{\Delta}$ and $t \in \delta_{\Delta}$ has distance three in $\Delta(G)$, while all the other
pairs of vertices have distance at most two, and the edges linking $\pi_0$ and $\pi_1$ have one neighbor in
$\beta_{\Delta} = \pi_0 \setminus \alpha_{\Delta}$ and the other in $\gamma_{\Delta} = \pi_1 \setminus \delta_{\Delta}$. See Figure~\ref{figure1}.

With respect to this notation, in the example $\Delta = \Delta(G)$ given in ~\cite{lewis1} we have $|\alpha_{\Delta}| = |\beta_{\Delta}| =|\delta_\Delta|= 1$ and
$|\gamma_{\Delta}| = 3$. See Figure~\ref{figure2}.

To the best of our knowledge, the example in~\cite{lewis1}  appears to be the only example of a character degree graph  of a solvable group, having diameter three,  published till now.
In Section~\ref{construction}, building on the ideas of~\cite{isaacs1} and~\cite{lewis1},  we  construct a class of examples of solvable groups that yield the following result.

\begin{figure}[!ht]
\begin{tikzpicture}
\fill (0,0) circle (2pt);
\fill (.5,0) circle (2pt);
\fill (1,0) circle (2pt);
\fill (-1,0) circle (2pt);
\fill (-.5,.5) circle (2pt);
\fill (-.5,-.5) circle (2pt);
\draw (0,0)--(0.5,0);
\draw (0.5,0)--(1,0);
\draw (-1,0)--(0,0);
\draw (-1,0)--(-.5,.5);
\draw (-1,0)--(-.5,-.5);
\draw (.5,0)--(-.5,.5);
\draw (.5,0)--(-.5,-.5);
\draw (0,0)--(-.5,.5);
\draw (0,0)--(-.5,-.5);
\draw (-.5,-.5)--(-.5,.5);
\end{tikzpicture}
\caption{Lewis' example}\label{figure2}
\end{figure}
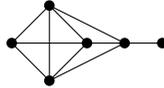

\begin{thmA}
  For every choice of positive integers $a, b$, there exists a solvable group $G$ such that $\Delta = \Delta(G)$ has diameter three 
  with $|\alpha_{\Delta}| =  a$ and  $|\delta_{\Delta}| \geq b$.
\end{thmA}

We prove also

\begin{thmB}
  Let $\Delta = \Delta(G)$  for a solvable group $G$ and assume that $\Delta$ has diameter three.
  Then
  $$|\gamma_{\Delta}| \geq 2^{|\beta_{\Delta}|}\left( 2^{|\alpha_{\Delta}|} -1  \right) + 1 .$$
\end{thmB}

We believe that $|\beta_\Delta|$ is related to the nilpotency class of $P$.
\begin{question}\label{q111}
Let $G$ be a solvable group such that $\Delta(G)$ has diameter $3$.  Is it true that $|\beta_\Delta|$ is at most the number of primes less than then nilpotency class of the non-abelian normal Sylow subgroup of $G$?
\end{question}

Actually, we are not aware of any examples with $|\beta_{\Delta}| > 1$. We remark that, for all the character graphs $\Delta = \Delta(G)$ of the solvable groups in Section~\ref{construction}, and then in Theorem~A,
as well as in Lewis' example~\cite{lewis1}, the set $\beta_{\Delta}$ consists of a single vertex, namely the prime $3$,  which is then a \emph{cut-vertex}
of $\Delta$.

\begin{question}\label{q1}
  Are there solvable groups $G$ such that $\Delta = \Delta(G)$ has diameter three and $|\beta_{\Delta}| \geq 2$?
\end{question}

The character degree graphs possessing  a cut-vertex have been studied in~\cite{HHHI} and \cite{LM} for solvable groups, and they
have been fully described for non-solvable groups in~\cite{DPSS}.
In particular, when $G$ is non-solvable, $\Delta(G)$ has a cut-vertex and diameter three if and only if
$G = J_1 \times A$, where $J_1$ is the first Janko group and  $A$ is an abelian group (\cite[Theorem~A(d)]{DPSS}).

Question~\ref{q1} is quite challenging at present, as our existing methods do not appear sufficiently effective. Furthermore, our current intuition is  heavily biased towards the known natural sources of examples. In light of this, we propose the following alternative, possibly more manageable question.
\begin{question}\label{q1a}
  Let $\Delta = \Delta(G)$ for a solvable group $G$ with $\Delta$ having diameter $3$. Is $\lim_{|\pi_0|\to\infty}\frac{|\alpha_\Delta|}{|\pi_0|}=1$?
\end{question}

Inspired by our work, we also pose the following.
\begin{question}\label{q1b}
  Let $\Delta = \Delta(G)$ for a solvable group $G$ with $\Delta$ having diameter $3$. Is $\lim_{|\pi_1|\to\infty}\frac{|\delta_\Delta|}{|\pi_1|}=0$?
\end{question}

In Theorem~\ref{primesets} we prove that for a character degree graph $\Delta = \Delta(G)$ of diameter three, where $G$ is a solvable group,
the union of the set $\beta_{\Delta}$ with a relatively large portion of the set $\gamma_{\Delta}$ induces a complete subgraph of $\Delta$. We leave open the following.

\begin{question}\label{q2}
  Let $\Delta = \Delta(G)$ for a solvable group $G$. If $\Delta$ has diameter three, then is the subgraph induced on
$\beta_{\Delta} \cup \gamma_{\Delta}$ a complete graph?
\end{question}

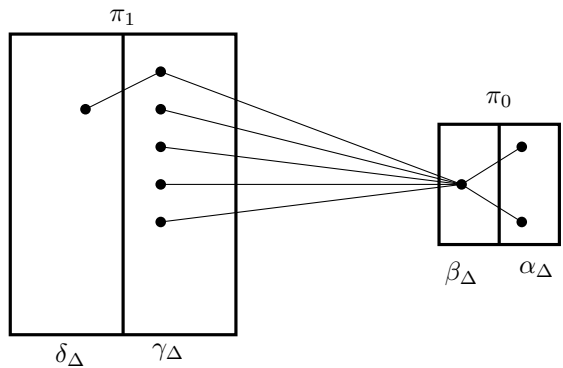
\begin{figure}[!ht]
\begin{tikzpicture}
\node[above] at (-2.5,2){$\pi_1$};
\node[below] at (-3.2,-2){$\delta_\Delta$};
\node[below] at (-1.9,-2){$\gamma_\Delta$};
\node[below] at (3,-.9){$\alpha_\Delta$};
\node[below] at (2,-.9){$\beta_\Delta$};
\draw[very thick] (-4,-2) rectangle (-1,2);
\draw[very thick] (-2.5,-2)--(-2.5,2);
\draw[very thick] (1.7,-.8) rectangle (3.3,.8);
\draw[very thick] (2.5,-.8)--(2.5,.8);
\node[above] at (2.5,.9){$\pi_0$};
\fill (2,0) circle (2pt);
\fill (2.8,.5) circle (2pt);
\fill (2.8,-.5) circle (2pt);
\fill (-3,1) circle (2pt);
\fill (-2,1.5) circle (2pt);
\fill (-2,1) circle (2pt);
\fill (-2,.5) circle (2pt);
\fill (-2,0) circle (2pt);
\fill (-2,-.5) circle (2pt);

\node[above] at (2,0){};
\draw (2,0)--(-2,1.5);
\draw (2,0)--(-2,1);
\draw (2,0)--(-2,.5);
\draw (2,0)--(-2,0);
\draw (2,0)--(-2,-.5);
\draw (2,0)--(2.8,.5);
\draw (2,0)--(2.8,-.5);
\draw (-3,1)--(-2,1.5);
\node[right] at (2.8,.5){};
\node[right] at (2.8,-.5){};
\end{tikzpicture}
\caption{An auxiliary picture for $\Delta(G)$ having diameter $3$}\label{figure1}
\end{figure}

\section{Preliminaries}\label{preliminaries}
This section serves as a compilation of notation and fundamental results employed consistently throughout the paper, often applied without explicit reference.

Given an action of a group $A$ on a set  $X$, we say that $X$ is an \emph{$A$-set}.
The action of $A$ on $X$ is \emph{semi-regular} if all $A$-orbits in $X$ have cardinality $|A|$, i.e. are regular orbits.

If $A$ acts by automorphisms on a group $G$ and $A$ acts semi-regularly on  $G \setminus \{1\}$, we say that
$A$ acts \emph{fixed-point-freely} on $G$.

Two $A$-sets $X$ and $X_1$ are said to be
\emph{isomorphic} if there exists a bijection $f:X \to X_1$ that commutes with the action of $A$.
If this happens when $A$ acts by automorphisms on the groups $X$ and $X_1$, we say that $X$ and $X_1$  are isomorphic $A$-groups.

We denote by $\cl G$ the set of the conjugacy classes of the group $G$.  For $g \in G$, we write $g^G$ for the $G$-conjugacy class  of $g$.
If a group $A$ acts via automorphisms on a group $G$, then $A$ acts naturally on $\irr G$ and on $\cl G$ as follows:
for $a \in A$, $g^G \in \cl G$, $\chi \in \irr G$, $x \in G$:
$$ \chi^a(x) = \chi(x^{a^{-1}}) \qquad \text{and} \qquad (g^G)^a = (g^a)^G . $$

We denote by $I_A(\chi) = \{ a\in A\mid \chi^a = \chi\}$ the \emph{inertia subgroup} of $\chi$ in $A$, i.e. the stabilizer
of $\chi$ under the action of $A$ on $\irr G$.
The following facts will be used repeatedly.

\begin{lemma}[\mbox{\cite[Theorem 13.24]{isaacs}}] \label{iso}
  Let $A$ be a solvable group acting by automorphisms on a group $G$.
  If $|A|$ is coprime to $|G|$, then $\irr G$ and $\cl G$ are isomorphic $A$-sets.
\end{lemma}

If $M$ is an abelian group, we denote by $\widehat M$ the \emph{dual group} of $M$, which is the set $\irr M$ with the natural product operation.
If $M$ is an abelian group such that $A$ acts by automorphisms and   such that $|A|$ is coprime to $|M|$, then as a particular case
of Lemma~\ref{iso} we have that $M$ and $\widehat{M}$ are isomorphic $A$-groups.

Given a normal subgroup $N$  of a group $G$, $G$ acts naturally by conjugation on $\irr N$ and $\cl N$.
For $\theta \in \irr N$, we use the following notation:
$$\irr{G|\theta} = \{ \chi \in \irr G | [\chi_N, \theta] \neq 0 \}$$
and, for a subset $Y \subseteq \irr N$,
$$\irr{G|Y} = \bigcup_{ \theta \in Y} \irr{G|\theta}.$$
Moreover, we write $\irr{G|N} =  \irr{G|\irr N \setminus \{1_N\}}$;
so, $\irr G = \irr{G/N} \dot\cup \irr{G|N}$.
\begin{lemma}\label{lemma1}
  Let $N$ be a normal subgroup of the group $G$ and $\theta, \theta_1 \in \irr N$. Then
  $\irr{G|\theta} = \irr{G|\theta_1}$
  if and only if $\theta$ and $\theta_1$ lie in the same $G$-orbit in $\irr N$.

  If $\mathcal{X}$ is a set of representatives for the non-empty intersections of the $G$-orbits in $\irr N$ with $X$, where $X$ is a subset of $\irr N$, then
  $$\irr{N|X} = \bigcup_{\theta \in \mathcal{X}}\irr{N|\theta}$$
  is a disjoint union.
 
\end{lemma}
\begin{proof}
  For $g \in G$, $\chi \in \irr G$ and $\theta \in \irr N$,
  $$[\chi_N, \theta^g] = [\chi_N^g, \theta^g] = [\chi_{N}, \theta].$$
  So, if $\theta_1 = \theta^g$ for some $g\in G$, then $\irr{G|\theta} = \irr{G|\theta_1}$.

  For $\chi \in \irr G$, the irreducible constituents of $\chi_N$ form a $G$-orbit in $\irr N$ by Clifford's theorem, so
   $\irr{G|\theta} \cap \irr{G|\theta_1} \neq \emptyset$ implies that $\theta$ and $\theta_1$ are $G$-conjugate, and the rest of the statement follows.
\end{proof}

We recall that a character $\tau \in \irr M$, where $M$ is a normal subgroup of a group $G$, is \emph{fully ramified with respect to
  $G/M$} if $\tau$ is $G$-invariant and $|\irr{G| \tau}| = 1$.
Observe that, given a group $A$ that  acts by automorphisms on $G$ and an $A$-invariant normal subgroup $N$ of $G$, if $\tau$ is fully ramified with respect to $G/N$
and $\irr{G|\tau} = \{\chi \}$, then the stabilizers in $A$ of $\tau$ and $\chi$ coincide.

\begin{lemma}\label{lemma2}
  Let $N$ be a normal subgroup of a group $G$ and $\theta \in \irr N$ such that $\theta$ is $G$-invariant.
  If $G/N$ is abelian, then there exists a unique subgroup $U$ such that $N \leq U \leq G$ and such that
 every $\tau \in \irr{U| \theta}$ is a  $G$-invariant extension of $\theta$ and $\tau$ is fully ramified with respect to $G/U$.

 Moreover, if the group $A$ acts by automorphisms on $G$ and both   $N$ and $\theta$ are $A$-invariant, then $U$ is $A$-invariant.
\end{lemma}
\begin{proof}
  This is a slight rephrasing of Lemma 2.2 of \cite{W}.
\end{proof}

\section{A construction}\label{construction}
In the following, for a positive integer $n$, we denote by $\pi(n)$ the set of the distinct prime divisors of $n$ and,
for a set of primes $\sigma$, we denote by $n_{\sigma}$ (or simply by $n_p$ if $\sigma = \{ p\}$) the $\sigma$-\emph{part} of $n$, that is the largest divisor $k$ of $n$ such that $\pi(k) \subseteq \sigma$.

Let $p$ be a prime number and let $n$ be an odd positive integer such that  $n \neq p$,  $n_3 = 3$  and
\begin{equation}\label{condition}\gcd\left(p^n-1,n\right)=1.
\end{equation}

Following~\cite{isaacs1}, let $F$ be a field of cardinality $p^n$ and let $F\{X\}$ be the skew polynomial ring with
$$Xa=a^pX,$$
for every $a\in F$.
 Next, let $R$ be the quotient of $F\{X\}$ by the ideal generated by $X^4$ and let $x$ be the image of $X$ in $R$. Therefore, $R=\{a_0+a_1 x+a_2 x^2+a_3 x^3\mid a_0,a_1,a_2,a_3\in F\}$.
Let $J$ be the Jacobson radical of $R$, that is, $J=\{a_1x+a_2x^2+a_3x^3\mid a_1,a_2,a_3\in F\}$. Now,
$$P=1+J,$$
is a subgroup of the group of units of $R$, $|P|=p^{3n}$ and each element of $P$ is of the form
$1+a_1x+a_2x^2+a_3x^3$, for some $a_1,a_2,a_3\in F$.

Let $C$ be the subgroup of the multiplicative group $F\setminus\{0\}$ of the field $F$ having order $(p^n-1)/(p-1)$. Clearly, $C$ is cyclic. The group $C$ acts as a group of automorphisms on $P$ by letting (see~\cite{HO})
\begin{align}\label{eq1}
(1+a_1x+a_2x^2+a_3x^3)^c
&=1+a_1c^{\frac{p-1}{p-1}}x+a_2c^{\frac{p^2-1}{p-1}}x^2+a_3c^{\frac{p^3-1}{p-1}}x^3\\\nonumber
&=1+a_1cx+a_2c^{p+1}x^2+a_3c^{p^2+p+1}x^3,
\end{align}
for every $a_1,a_2,a_3\in F$ and $c\in C$.
Let $T=P\rtimes C$. Therefore, $T$ is a group of order $(p^n-1)p^{3n}/(p-1)$.

Let
$H=\mathrm{Gal}(F/F_p)$,   where $F_p$ is the prime subfield of $F$; so, $H$ is a cyclic group  of order $n$.
Now, the group $H$ acts as a group of automorphisms  on $T$ by letting
\begin{equation}\label{eq2}
((1+a_1x+a_2x^2+a_3x^3) c)^\sigma=(1+a_1^\sigma x+a_2^\sigma x^2+a_3^\sigma x^3)  c^\sigma,
\end{equation}
for every $a_1,a_2,a_3\in F$, $c\in C$ and $\sigma\in H$.

Let $$G=T\rtimes H=P\rtimes (C \rtimes H).$$ Therefore, $G$ is a group of order $n(p^n-1)p^{3n}/(p-1)$ and
$CH$ is a complement of $P$ in $G$.

Recall that $n_3 = 3$. Let $D$ be the  subgroup of $C$ having order $p^2+p+1=(p^3-1)/(p-1)$.
\begin{lemma}\label{yet:another}The number $(p^3-1)/(p-1)$ and $(p^n-1)/(p^3-1)$ are relatively prime.

Moreover, $p-1$ is relatively prime to $(p^n-1)/(p-1)$.
\end{lemma}
\begin{proof}We have $(p^3-1)/(p-1)=p^2+p+1$ and
\begin{align*}
\frac{p^n-1}{p^3-1}&=p^{n-3}+p^{n-6}+\cdots+p^3+1\equiv \underbrace{1+1+\cdots+1+1}_{n/3\, \textrm{ times}}\pmod{p^3-1}.
\end{align*}
This shows that, if $d$ divides $p^2+p+1$ and $(p^n-1)/(p^3-1)$, then $d$ divides $n/3$. Since, by hypothesis, $\gcd(p^n-1,n)=1$, we deduce $d=1$.

Arguing as above, we have
\begin{align*}
\frac{p^n-1}{p-1}&=p^{n-1}+p^{n-2}+\cdots+p+1\equiv \underbrace{1+1+\cdots+1+1}_{n\, \textrm{ times}}\pmod{p-1}.
\end{align*}
Hence $\gcd(p-1,(p^n-1)/(p-1)) = \gcd(p-1,n)=1$, because $\gcd(p^n-1,n)=1$ by~\eqref{condition}.
\end{proof}

By Lemma~\ref{yet:another}, we have $$C=D\times E,$$ $|D|=p^2+p+1$ and $E$ is a Hall $\pi(p^2+p+1)'$-subgroup of $C$, $|E|=(p^n-1)/(p^3-1)$. Moreover, $F\setminus\{0\}=(F_p\setminus\{0\})\times D \times E$.

Following~\cite{isaacs1} , we let $P_i=1+J^i$. In particular, $P_1=P$, $P_4=1$ and $P_i/P_{i+1}$ is isomorphic to the additive group of the field $F$, for each $i\in  \{1,2,3\}$. Clearly, $P_i\unlhd G$, for each $i$.

\begin{lemma}\label{lemma:frobenius}
  \begin{description}
  \item[(a)] $T/P_3$ is a Frobenius group with Frobenius kernel $P/P_3$ and with Frobenius complement $CP_3/P_3\cong C$.
  \item[(b)]  $E$ acts fixed-point-freely on $P$ and, for every $E$-invariant subgroup $N$ of $P$ and a non-principal character $\theta \in \irr N$, we have
    $I_E(\theta) = 1$. Moreover, $D \leq \cent G{P_3}$.
  \end{description}
\end{lemma}
\begin{proof}
  (a): Let $1+a_1x+a_2x^2\in P$ and $c\in {\bf C}_C((1+a_1x+a_2x^2)P_3)$.  By~\eqref{eq1}, we have $$1+a_1cx+a_2c^{p+1}x^2=(1+a_1x+a_2x^2)^c\equiv 1+a_1x+a_2x^2\pmod {P_3}$$
if and only if $a_1 c=a_1$ and $a_2c^{p+1}=a_2$. Assume $c\ne 1$. We immediately deduce $a_1=0$. Moreover, since $p\equiv -1\pmod{p+1}$ and since $n$ is odd, we have
\begin{align*}
\frac{p^{n}-1}{p-1}&=p^{n-1}+p^{n-2}+\cdots+p+1\equiv(-1)^{n-1}+(-1)^{n-2}+\cdots+(-1)^1+(-1)^0\pmod{p+1}\\
&\equiv 1\pmod {p+1}.
\end{align*}
Therefore, $\gcd((p^n-1)/(p-1),p+1)=1$. In particular, $c^{p+1}\ne 1$ and hence $a_2=0$. Therefore,  $1+a_1x+a_2x^2= 1$ and part (a) is proved.

\smallskip
(b): Let $c\in C$. As $C= D\times E$, we may write $c=de$ for some $d\in D$ and $e\in E$. Let $t=1+b x^3\in P_3$. Then
$$t^c=1+b c^{p^2+p+1}x^3=1+b e^{p^2+p+1} x^3.$$
Then $D \leq \cent G{P_3}$ and, as $\gcd(p^2+p+1,|E|)=1$, $E$ acts fixed-point-freely on $P_3$.
  Since $E$ acts fixed-point-freely on
$P/P_3$ by part (a), $E$ acts fixed-point-freely on $P$.
Thus, if  $N$ is an $E$-invariant subgroup of $P$, then  $E$ acts fixed-point-freely on $N$,  and $\irr N$ and $\cl N$ are isomorphic
$E$-sets by Lemma~\ref{iso}.
Hence, the rest of part (b) follows by Glauberman's lemma (\cite[Lemma~13.8]{isaacs}).
\end{proof}

The map $(x,y)\mapsto [x,y]$ defines a function $$P/P_2\times P_2/P_3\to P_3.$$
Since $P/P_2$, $P_2/P_3$ and $P_3$ are naturally isomorphic to the additive group of the field $F$, this commutator mapping defines a function
\begin{equation}\label{eq3}\langle\cdot,\cdot \rangle:F\times F\to F.
\end{equation} This function is computed explicitly in~\cite{isaacs1}. Indeed,
let $s=1+a x+i\in P$ with $i\in J^2$ and let $t=1+b x^2+j\in P_2$ with $j\in J^3$. From~\cite[Corollary~4.2]{isaacs1}, we have
\begin{align}\label{eq:5}
[s,t]&=1+(ab^p-a^{p^2}b)x^3.
\end{align}
This shows that the function in~\eqref{eq3} is defined by $\langle a,b\rangle=ab^p-a^{p^2}b$.

From its definition, we deduce that, for every $a\in F$, $b\mapsto \langle a,b\rangle$ is a linear $F_p$-map, where (as above) $F_p$ is the finite subfield of $F$ having cardinality $p$.  We let
$$\langle a,F\rangle=\{ab^p-a^{p^2}b\mid b\in F\}$$
denote the image of this mapping.
\begin{lemma}\label{yet:anotherr}Let $a\in F$. If $a=0$, then $b\mapsto\langle a,b\rangle$ is the zero function. If $a\ne 0$, then the kernel of the function $b\mapsto \langle a,b\rangle$ is $\{c a^{p+1}\mid c \in F_p\}$ and the image $\langle a,F\rangle$ is an $F_p$-subspace of $F$ of codimension $1$.
\end{lemma}
\begin{proof}
The result is clear  when $a=0$. Assume then $a\ne 0$. Let $b\in F$ with $\langle a,b\rangle=0$. Thus $ab^p=a^{p^2}b$ implies $b^p=a^{p^2-1}b$, that is, $b\in \{c a^{p+1}\mid c \in F_p\}$.
\end{proof}

We let
\begin{equation}\label{eq:3}
  \pi_0 = \langle 1, F \rangle =\{b-b^p\mid b\in F\}.\footnote{A slight notation conflict arises in this context. Specifically, in Section~\ref{sec:intro}, $\pi_0$ is utilized to represent the connected component of $\Delta(G/\gamma_3(P))$ that does not include the prime $p$. It's worth noting that this discrepancy should not lead to confusion, as the context makes the distinction evident.}
\end{equation}

We note that, by the additive form of Hilbert's Theorem 90, $\pi_0$ is the kernel of the trace map $\mathrm{Tr}: F \rightarrow F_p$, so  $\pi_0$ is an $H$-invariant hyperplane of $F$ (seen as an $F_p$-vector space).

Let $\Pi_0$ be the $E$-orbit, with respect to the natural action of $E$ on the hyperplanes of $F$, containing $\pi_0$.
We denote by $F_{p^3}$ the subfield of $F$ of cardinality $p^3$.

\begin{lemma}\label{lemma:te1}
For $a \in F$, we have $\langle a, F \rangle = \pi_0$ if and only if $0 \neq a \in F_{p^3}$, and
 $$\Pi_0 = \{\langle a,F\rangle\mid a\in F\setminus\{0\}\}.$$
\end{lemma}
\begin{proof}
Let $a\in F\setminus\{0\}$. As $F\setminus\{0\}=(F_p\setminus\{0\})\times D\times E$, we may write $a=a_f a_da_e$ for some $a_f\in F_p\setminus\{0\}$, $a_d\in D$ and $a_e\in E$.
Next, let $b\in F\setminus\{0\}$ and let $b_0=a_e^{-(p+1)}b$. We have
\begin{align*}
\langle a,b\rangle&=
ab^p-a^{p^2}b=
a_fa_da_eb_0^pa_e^{p^2+p}-a_f^{p^2}a_d^{p^2}a_e^{p^2}b_0a_e^{p+1}\\
&=a_fa_db_0^pa_e^{p^2+p+1}-a_fa_d^{p^2}b_0a_e^{p^2+p+1}\\
&=a_e^{p^2+p+1}\left(a_fa_db_0^p-a_fa_d^{p^2}b_0\right)\\
&=a_e^{p^2+p+1}\underbrace{\left((-a_fa_d^{p^2}b_0)-(-a_fa_d^{p^2}b_0)^p\right)}_{\in \pi_0}\in a_e^{p^2+p+1}\pi_0.
\end{align*}
Hence, recalling Lemma~\ref{yet:anotherr}
$$\langle a , F \rangle = a_e^{p^2+p+1}\pi_0$$
and, as $|E|$ is coprime to $p^2 + p +1$ by Lemma~\ref{yet:another}, we conclude that $\langle a, F \rangle = \pi_0$
if and only if $a_e = 1$, which is equivalent to $a \in (F_p\setminus\{0\})\times D = F_{p^3}\setminus\{0\}$.  

Finally, by observing that, as $a_e$ runs through the elements of $E,$  so does $a_e^{p^2+p+1}$, we 
see  that the hyperplanes of the form $\langle a,F\rangle$ are in one to one correspondence with the elements of $E$ and,
in particular, that $\Pi_0 = \{\langle a,F\rangle\mid a\in F\setminus\{0\}\}$.
\end{proof}

\smallskip
Write $$H = R \times S,$$ where $R$ is the Sylow $3$-subgroup of $H$,  so $|R| = 3$, and
$S$ is the $3$-complement of $H$.

\smallskip
As $P_3=\{1+a x^3\mid a\in F\}$,
the mapping $\omega: P_3\to F$ defined by $1+a x^3\mapsto a$ sets up a group isomorphism
between  $P_3$ and the additive group of $F$. We also remark that, via $\omega$,  $P_3$ and $F$ are isomorphic $H$-sets.
\begin{lemma} \label{lce}
  Let $X = \widehat{P_3}\setminus\{1_{P_3}\}$. Then $G$ acts on $X$  and we denote by $X_0$  the union of the $H$-invariant $C$-orbits in $X$. Then
  \begin{description}
  \item[(a)] For every $\varphi \in X \setminus X_0$, there exists  $g \in G$ such that $I_G(\varphi^g) = PDS$.
  \item[(b)] If $\varphi_0 \in X$ is such that $\omega(\mathrm{ker}(\varphi_0)) = \pi_0$, then $\{ \varphi_0^i \mid i = 1, \ldots, p-1\}$
    is a set of representatives  for the $G$-orbits in $X_0$. Moreover, $I_G(\varphi_0^i) = PDH$ for every $i = 1, \ldots, p-1$.
  \item[(c)] For every $\varphi \in X$, $\varphi\in X_0$ if and only if  $\omega(\ker(\varphi)) \in \Pi_0$.
  \end{description}
\end{lemma}
\begin{proof}
Since $P_3 = \mathrm{Z}(P)$ and $D \leq \cent G{P_3}$ by Lemma~\ref{lemma:frobenius} part~(b),  $PD \leq \cent G{P_3}$. Therefore, as $G=PDEH$, the $G$-orbits and the $EH$-orbits in $X$ coincide.
  Let  $X' = P_3 \setminus \{1\}$.
With reference to Lemma~\ref{iso}, it is established that $X$ and $X'$ are isomorphic $EH$-sets. Consequently, we leverage the action of $EH$ on $X'$ to draw conclusions regarding the action of $EH$ on $X$.

Since the map $E \to E$ such that $e \mapsto e^{p^2+p+1}$ is a bijection, the orbits of $E$ on $X'$ are the subsets of the form
$\{1 + ea x^3 | e \in E\}$, for  $0 \neq a \in F_{p^3}$. So  they correspond, identifying   $X'$ and $F^\times = F \setminus \{ 0\}$ via $\omega$,  to
the non-trivial cosets  of $E$ in $F^\times$, i.e. to  the non-trivial elements of the  group $F^\times/E \simeq F_{p^3}^\times$.
Let $\sigma  \in H$ be the cube of the Frobenius automorphism of $F$;   so $S = \langle \sigma \rangle$. Since $a^{\sigma} = a^{p^3} = a^{p^3 -1} a$, and $a^{p^3 -1} \in E$ (as $o(a^{p^3 -1})$ divides $(p^n-1)/(p^3 -1)= |E|$), it follows that $S$ fixes all $E$-orbits on $X'$. On the other hand, $R$ acts on them as it acts on
$F_{p^3}^\times$, so $R$ fixes exactly $p-1$ $E$-orbits on $X'$.
Therefore, by the isomorphism of the $EH$-sets $X$ and $X'$,  $S$ acts trivially on set of the $E$-orbits in
$X$ and $R$ fixes exactly $p-1$ of them (and permutes the others in orbits of size $3$).
Hence,  $X_0$ contains exactly $p-1$ $G$-orbits.
Moreover, by Glauberman's lemma every $E$-orbit in $X \setminus X_0$ contains an
$S$-invariant character.  So, recalling that, by Lemma~\ref{lemma:frobenius} part~(b),  $I_E(\varphi) = 1$ for every
$\varphi \in X$, part (a) follows.
We remark that, hence, the $C$-orbits in $X$ are in fact $G$-orbits.

We observe that for $\varphi \in X$, the stabilizer of the group $\langle\varphi\rangle$ generated by $\varphi$ under the action of  $E$ is trivial,  because $\cent{E}{\langle\varphi\rangle}=1$
and $|\mathrm{Aut}(\langle \varphi \rangle)| = p-1$ is coprime to  $|E|$.
It follows that $\langle \varphi \rangle$ intersects $p-1$ distinct $E$-orbits in $X$.

Let now $\varphi_0 \in X$ be such that $\omega(\mathrm{ker}(\varphi_0)) = \pi_0$. Then $\varphi_0$ is $H$-invariant.
In fact, writing $H = \langle h \rangle$, where $h$ is the Frobenius automorphism of $F$,  for any $t = 1 +b x^3 \in P_3$,
$\varphi_0^{h^{-1}}(t) = \varphi_0(t^h) = \varphi_0(t)$ as $t^ht^{-1} = (1+b^p x^3)(1 - b x^3) = 1+ (b^p - b)x^3 \in
\mathrm{ker}(\varphi_0)$. Thus, by the previous paragraph, we conclude that $\{ \varphi_0^i \mid i = 1, \ldots, p-1\}$
is a set of representatives  for the $G$-orbits in $X_0$. Since the characters $\varphi_0^i$ generate the same
subgroup of $\widehat{P_3}$ for every $i = 1, \ldots, p-1$, recalling again that $I_E(\varphi_0) = 1$, we have
that $I_G(\varphi_0^i) = PDH$ for every $i = 1, \ldots, p-1$, and part (b) is proved.

Finally, recalling the definition of $\Pi_0$ and $X_0$, we deduce that  $\varphi \in X_0$ if and only if  $\omega(\ker(\varphi)) \in \Pi_0$, so we have part (c).
\end{proof}

Let $B=\{1+b x^2\mid b\in F\} \leq  P_2$. Observe that $B$ is normalized by $CH$ and $$P_2=B\times P_3.$$
Therefore, the characters in $\mathrm{Irr}(P_2| P_3)$ are of the form $\nu\times \varphi$, where $\nu\in\mathrm{Irr}(B)$ and $\varphi\in X$. Recall that, from Lemma~\ref{lce}, $X = \widehat{P_3}\setminus\{1_{P_3}\}$.

Let
$$Q=\{1+a_1 x+a_2x^2+a_3x^3\in P\mid a_1,a_2,a_3\in F_{p^3}\} = \cent PS.$$

Given $\varphi\in X$ and $a\in F$, we define $\varphi_a\in\mathrm{Irr}(B)$ by setting
\begin{equation}\label{eq:def}\varphi_a(1+b x^2)=\varphi(1 -\langle a,b\rangle x^3),\end{equation}
for every $b\in F$.

\begin{lemma}\label{lemma:tech2}
  Let $\mu\times \varphi\in\mathrm{Irr}(P_2)$, with $\mu\in\mathrm{Irr}(B)$ and $\varphi\in X$. Then
  \begin{description}
  \item[(a)]   for every $s=1+a x+i\in P$ with $i\in J^2$, we have
$$(\mu\times \varphi)^{s^{-1}}=\mu\varphi_a\times \varphi.$$

\item[(b)] If $\varphi_0$ and $X_0$ are as in Lemma~$\ref{lce}$, then $I_P(\mu\times \varphi_0) = P_2 Q$. Moreover,  for
  $\varphi  \in X \setminus X_0$, we have $I_P(\mu\times \varphi) = P_2$.
  \end{description}
\end{lemma}

\begin{proof}
 Let $t=1+b x^2+cx^3\in P_2$. Then,
using~\eqref{eq:5}, we get
\begin{align*}
t^s&=t[t,s] = t[s,t]^{-1}=(1+b x^2+cx^3)(1 -\langle a,b\rangle x^3)\\
&=1+b x^2+(c- \langle a,b\rangle )x^3.
\end{align*} Therefore, we have
\begin{align}\label{eq:saturday1}
(\mu\times \varphi)^{s^{-1}}(t)&=(\mu\times \varphi)(t^s)=(\mu\times \varphi)(1+b x^2+(c-\langle a,b\rangle )x^3)\\\nonumber
&=(\mu\times \varphi)((1+b x^2)(1+(c-\langle a,b\rangle )x^3))
=\mu(1+b x^2)\varphi(1+(c-\langle a,b\rangle )x^3).
\end{align}
Similarly,
\begin{align}\label{eq:saturday2}
(\mu\varphi_a\times \varphi)(t)&=\mu\varphi_a(1+b x^2)\varphi(1+cx^3)=\mu(1+b x^2)\varphi(1-\langle a,b\rangle x^3)\varphi(1+c x^3)\\\nonumber
&=\mu(1+b x^2)\varphi(1+(c-\langle a,b\rangle)x^3).
\end{align}

Now part (a) follows immediately from~\eqref{eq:saturday1} and~\eqref{eq:saturday2}.

Moreover, writing  $s=(1+a x+i)w$ with $w \in P_2$, by part (a)  $s \in I_P(\mu \times \varphi)$ if and only if
$\langle a, F \rangle = \omega(\mathrm{ker}(\varphi))$, so part (b) follows from Lemma~\ref{lemma:te1} and part (c) of Lemma~\ref{lce}.
\end{proof}

Since  $\gcd(|S|,|B|)=1$ and $Q \cap B = \cent QS$, writing $B_1 = [B,S]$ we have
$B=(Q\cap B)\times B_1$.
Hence,
\begin{equation}\label{eq:direct}\widehat{B} = \widehat{Q\cap B}\times \widehat{B_1}. \end{equation}
In this way,  we see $\widehat{Q\cap B}$ and $\widehat{B_1}$ as  subgroups of $\widehat{B}$.

\begin{lemma}\label{lemmaY}
Let $\varphi_0\in \widehat{P_3}$ be as in Lemma~$\ref{lce}$. Then $PDH$ acts on
 $Y = \irr{P_2|\varphi_0}$.   Let $\psi_0 = 1_B \times \varphi_0$ and let $Y_0$ be the $PDH$-orbit of $\psi_0$ in $Y$.
Then
  \begin{description}
  \item[(a)] $\{(\varphi_0)_a\mid a\in F\} = \widehat{B_1}$.
    \item[(b)]For every $\psi \in Y$, there exists an element $g \in PD$ such that $\psi^g$ is $H$-invariant.
  \item[(c)]  $Y_0$ is a $P$-orbit and it is the  unique $D$-invariant $P$-orbit in $Y$.
  \item[(d)]  $D$ acts semi-regularly on the set of the $P$-orbits in $Y \setminus Y_0$.
  \item[(e)] $Y \setminus Y_0$ consists of exactly $p-1$ $PDH$-orbits.
  \end{description}
\end{lemma}
\begin{proof}
Since $\varphi_0$ is $PDH$-invariant by Lemma~\ref{lce} part~(b),  $PDH$ acts on  $Y$.
  As $P_2 = B \times P_3 $, we have   $Y = \{\mu \times \varphi_0 \mid \mu \in \widehat{B} \}$
and, as $B$ is $D$-invariant,  the action of $D$ on $Y$
is  isomorphic to the action of $D$ on $\widehat{B}$ and hence, by Lemma~\ref{iso} to the
action of $D$ on $B$.
Since $D$ acts fixed-point-freely on $B$, we deduce that
$$\psi_0 = 1_B \times \varphi_0$$
is the unique $D$-invariant character in $Y$.  As $H$ fixes $D$ and $Y$, this uniqueness implies that
$\psi_0$ is $H$-invariant as well. So, $\psi_0$ is $DH$-invariant and hence $Y_0$ is a $P$-orbit.
Moreover, if $Y_1$ is a $D$-invariant $P$-orbit in $Y$, then $Y_1$ contains a $D$-invariant character by
Glauberman's lemma, and hence $Y_1 = Y_0$, proving (c).

In order to prove (a), we first show that $Q\cap B\le\mathrm{Ker}((\varphi_0)_a)$ for every $a \in F$.
For $1 \neq 1+b x^2 \in Q \cap B$, so $b\in F_{p^3}^\times$, we have  $b^{p^3-1}=1$. Thus
\begin{align*}
\langle a,b\rangle&=ab^p-a^{p^2}b=(ab^p-(ab^p)^p)+(a^pb^{p^2}-(a^pb^{p^2})^p)\in \pi_0=\mathrm{Ker}(\varphi_0).
\end{align*}
This and~\eqref{eq:def} show that $1+b x^2\in \mathrm{Ker}((\varphi_0)_a)$.
Therefore, $\{(\varphi_0)_a\mid a\in F\} \subseteq \widehat{B_1}$.

On the other hand, the mapping $F\to\widehat{B}$ such that  $a\mapsto (\varphi_0)_a$  is a homomorphism of $F_p$-vector
spaces and its kernel is $F_{p^3}$ by Lemma~\ref{lemma:te1}. As $|\widehat{B_1}| = p^{n-3}$,  part (a) follows.

Now, part (a) of Lemma~\ref{lemma:tech2}, (\ref{eq:direct}) and part (a) of the statement yield that
\begin{equation}
  \label{eq:reps}
  \{\delta \times \varphi_0 \mid \delta \in \widehat{Q \cap B} \}
\text{  is a set of representatives for the $P$-orbits in $Y$.}
\end{equation}

Since $\varphi_0$ is $D$-invariant, by (\ref{eq:reps}) the action of $D$ on the set of  the $P$-orbits in $Y$ is
isomorphic to the action of $D$ on $\widehat{Q \cap B}$.
Since $D$ acts transitively on the $1$-dimensional subspaces of $Q\cap B$ (identified, as usual, with $F_{p^3}$),
and $\cent{Q\cap B}H$ is a $1$-dimensional subspace, it follows that every $D$-orbit in $Q \cap B$ contains an
$H$-invariant element and, by the isomorphism of the actions of $D$ on $Q\cap B$ and on $\widehat{Q \cap B}$, the same is true
for $\widehat{Q \cap B}$.
Therefore, we conclude that every $PD$-orbit in $Y$ contains an $H$-invariant character, proving part (b).

By the same argument, since   that $D$ acts fixed-point-freely on $\widehat{Q \cap B}$, we deduce that
  $D$ acts semi-regularly  on the $P$-orbits in $Y$ that do not contain $\psi_0$, so we have part (d).

  Finally, by Lemma~\ref{lemma:tech2} part~(b) $I_P(\psi) = P_2Q$ for every $\psi \in Y$. By part (b) it follows that
  the $PDH$-orbits in $Y$ are in fact $PD$-orbits and, as $|Y_0| = p^{n-3}$, by part~(d) of this statement
  there are $(|Y| - p^{n-3})/(p^{n-3}|D|) = p-1$ $PDH$-orbits in $Y\setminus Y_0$, completing the proof.
\end{proof}

\section{Number theoretic considerations}\label{thrm:number theory}
We start with a result concerning the existence of integers with suitable properties.

\begin{lemma}\label{number theory3}
For every positive integers $c$ and $\ell$, there exists a prime number $p$ and a positive integer $n$ such that
\begin{description}
\item[(a)] $p^2+p+1$ is divisible by at least $c$ distinct primes
\item[(b)]$\gcd(n,p(p^n-1))=1$
\item[(c)] $n$ is odd and  the largest power of $3$ dividing $n$ is $3$
\item[(d)] $n$ is divisible by exactly $\ell$ distinct primes.
\end{description}
\end{lemma}
\begin{proof}
First we show, by induction on $c$, that there exist $c$ distinct primes $q_1,\ldots,q_c$ and positive integers $m_1,\ldots,m_c$ with the property that $q_i$ divides $m_i^2+m_i+1$ and $q_i \neq 3$, for each $i$.
  Indeed, when $c=1$, we may choose any integer $m_1$ divisible by $3$ and a prime divisor $q_1$ of $m_1^2 + m_1 +1$.
  Now, suppose that $c>1$ and that $q_1,\ldots,q_{c-1}$ and $m_1,\ldots,m_{c-1}$ are defined.
  Let $m_c=3\cdot q_1\cdots q_{c-1}$ and let $q_c$ be a prime divisor of $m_c^2+m_c+1$. Clearly, $q_c\notin\{q_1,\ldots,q_{c-1}\}$ and $q_c \neq 3$.

We continue by  showing the existence of $\ell$ prime numbers $r_1,\ldots,r_\ell$  with $3=r_{1}<r_2<\cdots<r_\ell$ such that, for each $i\in \{2,\ldots,\ell-1\}$,
\begin{align}\label{today}
r_i&\equiv 1\pmod {(r_1-1)\cdots (r_{i-1}-1)},\\\label{tomorrow}
r_i&\equiv 2\pmod{r_1\cdots r_{i-1})}
\end{align}
and such that $$\{r_1,\ldots,r_\ell\}\cap \{q_1,\cdots,q_c\}=\emptyset.$$
We argue inductively on $\ell$. When $\ell=1$, we have $r_1=3$ and we have nothing more to show. Assume now that $i = \ell -1 \ge 1$ and that $r_1,\ldots,r_i$ have been defined and let $a=(r_1-1)\cdots(r_i-1)$ and $b=r_1\cdots r_i$. We claim that
\begin{equation}\label{eq:number1}
\gcd(a,b)=1.
\end{equation}
Let $r$ be a prime number with $r$ dividing $\gcd(a,b)$. By definition of $b$, we deduce that $r=r_j$ for some $j\in \{1,\ldots,i\}$. Since $r=r_j$ divides $a$, we deduce that $r_j$ divides $r_{j'}-1$ for some $j'\in \{1,\ldots,i\}$. Clearly  $j'>j$, and the  inductive hypotheses imply $r_{j'}\equiv 2\pmod{r_1\cdots r_{j'-1}}$. Hence $r_{j'}\equiv 2\pmod{r_{j}}$ and $r_{j'}-1\equiv 1\pmod{r_j}$, contradicting the fact that $r_j$ divides $r_{j'}-1$. These contradictions have established the veracity of~\eqref{eq:number1}.

From~\eqref{eq:number1}, we deduce the existence of $\alpha,\beta\in \mathbb{Z}$ with $\alpha a+\beta b=1$. Let $d=2\alpha a+\beta b$. We claim that
\begin{equation}\label{eq:number2}
\gcd(d,ab)=1.
\end{equation}
Let $r$ be a prime number with $r$ dividing $d$ and $ab$. As $r$ is prime, $r$ divides either $a$ or $b$. If $r$ divides $a$, then $r$ divides $d-2\alpha a=\beta b$. Therefore $r$ divides $\alpha a+\beta b=1$, which is a contradiction. If $r$ divides $b$, then $r$ divides $d-\beta b=2\alpha a$. Since $b=r_1\cdots r_i$ is odd and since $r$ divides $b$, we deduce that $r$ is odd. Therefore $r$ divides $\alpha a$. Thus $r$ divides  $\alpha a+\beta b=1$, which is a contradiction. This has established~\eqref{eq:number2}.

From~\eqref{eq:number2} and from Dirichlet's theorem on primes in an arithmetic progression, there exists a prime  $r_{i+1}$ such that
$r_{i+1}>r_i$, $r_{i+1} \neq q_j$ for all  $j \in \{ 1, \ldots,  c\}$ and  $r_{i+1}\equiv d\pmod{ab}$.
We have
\begin{align*}
r_{i+1}\equiv d= 2\alpha a+\beta b \equiv \beta b \equiv 1\pmod{a},
\end{align*}
where the last congruence follows from the fact that $\alpha a+\beta b=1$ and hence $\beta b\equiv 1\pmod a$. Similarly, as $r_{i+1}\equiv d\pmod{ab}$, we have
\begin{align*}
r_{i+1}\equiv d= 2\alpha a+\beta b \equiv 2\alpha a \equiv 2\pmod{b},
\end{align*}
where the last congruence follows from the fact that $\alpha a+\beta b=1$ and hence $2\alpha a\equiv 2\pmod b$. Since $a=(r_1-1)\cdots (r_i-1)$ and $b=r_1\cdots r_i$, this concludes our inductive proof of the existence of $r_1,\ldots,r_\ell$.

Now set $$n=r_1\cdots r_\ell$$
and observe that at this point part~(c) and~(d) are immediately satisfied.

From the Chinese Reminder theorem, let $m\in\mathbb{N}$ with $m\equiv m_i\pmod{q_i}$, for each $i\in \{1,\ldots,c\}$. Observe that $\gcd(m,q_1\cdots q_c)=1$, hence $\gcd(m(r_1\cdots r_{\ell}), q_1\cdots q_c) = 1$.
From Bezout's theorem, let  $\alpha,\beta\in \mathbb{Z}$ with $\alpha (q_1\cdots q_c)+\beta(r_1\cdots r_\ell)=1$ and set $$f=-\alpha(q_1\cdots q_c)+m\beta(r_1\cdots r_\ell).$$
Observe, arguing as above, that $\gcd(f,q_1\cdots q_c\cdot r_1\cdots r_\ell)=1$.
Hence, Dirichlet's theorem implies that there is  a prime $p$  (in fact,  infinitely many) such that 
\begin{equation}\label{eq:number13}p\equiv f\pmod {q_1\cdots q_c\cdot r_1\cdots r_\ell}.
\end{equation}
Working modulo $q_1\cdots q_c$, we have $p\equiv f\pmod{q_1\cdots q_c}$ and hence $p\equiv m\pmod{q_1\cdots q_c}$, because $f\equiv m\pmod{q_1\cdots q_c}$.
Hence, 
$$p^2+p+1\equiv m^2+m+1 \equiv m_i^2+m_i+1 \equiv 0\pmod{q_i}, $$
for each $i\in \{1,\ldots,c\}$, and  part~(a) is  satisfied.

Moreover, working modulo $r_1\cdots r_\ell$, we have $p\equiv f\pmod{r_1\cdots r_\ell}$ and hence
\begin{equation}\label{eq:number14}p\equiv -1\pmod{r_1\cdots r_\ell},\end{equation}
because $f\equiv -1\pmod{r_1\cdots r_\ell}$.
Since $n = r_1\cdots r_\ell$ is odd, $p^n\equiv -1\pmod {r_i}$ for all  $ i \in \{1, \ldots,  \ell \}$,  and hence
$\gcd(n, p(p^n-1))= 1$.
\end{proof}

Lemma~\ref{number theory3} suffices for our application, namely, the proof of Theorem~A. However, before finalizing this section, we wish to include a few remarks on the number-theoretic aspects pertaining to the condition in~\eqref{condition}.

Let $n$ be a positive integer. Also, let
$$x^n-1=\prod_{d\mid n}\lambda_d(x), \quad\lambda_d(x)=\prod_{\substack{z\in\mathbb{C}\\\textrm{primitive } d^{\mathrm{th}}\\ \textrm{root of unity}}}x-z$$
be the factorization of $x^n-1$ into its irreducible factors in $\mathbb{Q}[x]$ (i.e. cyclotomic polynomials). 

Now, let $q$ be an integer, $q \geq 2$, and let $d$ be a divisor of $n$ with $d\ne 1$. Observe that, if $z\in \mathbb{C}$ is a primitive $d^{\textrm{th}}$ root of unity, then the Euclidean distance between $z$ and $q$ is greater than $q-1$, that is, $|z-q|>q -1$.  We deduce
\begin{align*}
|\lambda_d(q)|&=\prod_{\substack{z\in\mathbb{C}\\\textrm{primitive } d^{\mathrm{th}}\\ \textrm{root of unity}}}|z-q| > \prod_{\substack{z\in\mathbb{C}\\\textrm{primitive } d^{\mathrm{th}}\\ \textrm{root of unity}}}(q-1) = (q-1)^{\varphi(d)}\ge 1.
\end{align*}
This shows that $\lambda_d(q)\ne\pm 1$.

For a positive integer $k$, we denote by
$d(k)$ the number  of positive divisors of $k$.
\begin{lemma}\label{lemma:nt} Let $n$ and $q$ be  positive integers with $q \geq 2$ and  $\gcd(n,(q^{n}-1)/(q-1))=1$.
  Then, for every positive integer $m$ that divides $n$,
  $$
\left|
\pi
\left(
\frac{q^{n}-1}{q^m-1}
\right)
\right|
\ge d(n) - d(m).$$ 
\end{lemma}
\begin{proof}
For each divisors $d$
of $n$ different from $1$, let $r_d$ be a prime divisor of $\lambda_d(q)$, where $\lambda_d$ is the $d$-th cyclotomic polynomial. Observe that the existence of $r_d$ follows from the fact that $\lambda_d(q)>1$.

We claim that $q$ has order exactly $d$ modulo $r_d$, that is, the order of the residue class $q+r_d\mathbb{Z}$ in $\mathbb{Z}/r_d\mathbb{Z}$ is $d$. From this, it follows that, when $d$ varies among the divisors $\ge 2$ of $n$, all primes $r_d$ are distinct and hence the result follows. Let $o$ be the order of $q$ modulo $r_d$. From the definition of order and from the factorization of $x^o-1$, it follows that $r_d$ divides $\lambda_o(q)$. As $r_d$ divides $\lambda_d(q)$ and $\lambda_d(q)$ divides $q^d-1$, we deduce that $q^d\equiv 1\pmod {r_d}$ and hence $o$ divides $d$.  Recall
$$x^d-1=\prod_{a\mid d}\lambda_a(x).$$Now, since $\gcd(n,(q^n-1)/(q-1))=1$, $d$ is relatively prime to $r_d$ and hence the derivative $(x^d-1)'=dx^{d-1}$ in the ring $\mathbb{Z}/{r_d}\mathbb{Z}[x]$ is relatively prime to $x^d-1$. As $x^d-1\in \mathbb{Z}/r_d\mathbb{Z}[x]$ has no multiple roots, it is not possible that $q$ is a root of $\lambda_d(x)$ and $\lambda_o(x)$ in $\mathbb{Z}/r_d\mathbb{Z}[x]$. Since $r_d$ divides both $\lambda_o(q)$ and $\lambda_d(q)$, this implies $o=d$.
%
\end{proof}

In view of the proof of  Lemma~\ref{lemma:nt}, the condition $\gcd(p^n-1,n)=1$ in~\eqref{condition} implies that, when $p\ge 3$,  the number of distinct prime divisors of $p^n-1$ is at least $2^{|\pi(n)|}$, because for each divisor $d$ of $n$ we may select from $\lambda_d(p)$ at least one prime divisor. However, when $p=2$, since $\lambda_1(p)=p-1=1$, we may only infer that the number of distinct prime divisors of $2^n-1$ is at least $2^{|\pi(n)|}-1$. Except when $\ell\le 2$, we have no examples where $2^n-1$ has exactly $2^{|\pi(n)|}-1$ distinct prime divisors. We dare to state the following.

\begin{conjecture}\label{conj:1}{\rm
Let $n$ be divisible by $\ell$ distinct primes and assume $\gcd(2^n-1,n)=1$. If $\ell\ge 3$, then $2^n-1$ is divisible by at least $2^\ell$ distinct primes.}
\end{conjecture}

\section{Main results}

We start by proving Theorem~A, that we state again.
\begin{thmA}
  For every choice of positive integers $a, b$, there exists a solvable group $G$ such that $\Delta = \Delta(G)$ has diameter three and
  such that $|\alpha_{\Delta}| =  a$ and  $|\delta_{\Delta}| \geq b$.
\end{thmA}

\begin{proof}
According to Lemma~\ref{number theory3}, we choose a prime number $p$ and a positive odd integer $n$ such that
$n$ is divisible by exactly $a$ distinct primes,  $p^2+p+1$ is divisible by at least $b$ distinct primes,
$\gcd(n,p(p^n-1))=1$, and  the largest power of $3$ dividing $n$ is $3$.

Let $G$ be the group defined, for $p$ and $n$ as above,   in  Section~\ref{construction}.
We are going to determine the degrees, along with their multiplicities, of the irreducible characters of $G$.
We will make consistent use of the notation introduced in  Section~\ref{construction}.

\medskip
{\bf (I)}~~$\irr{G/P}$

\smallskip
 As $C$ is an abelian normal subgroup of $CH$, an irreducible character of $G/P \simeq CH$ has degree $d$, where $d$ is a divisor of $n = |H|$. Let $a_d$ be the number of characters in $\mathrm{Irr}(CH)$ having degree $d$. Therefore,
$$\frac{(p^n-1)}{p-1}n=|CH|=\sum_{d\mid n}a_dd^2.$$
Using M\"{o}bius inversion formula, we deduce that $$a_dd^2=n\sum_{\ell\mid d}\mu\left(\frac{d}{\ell}\right)\frac{p^\ell-1}{p-1}.$$
Therefore, the number of characters of degree $d$ in $G/P$ is
$$a_d=\frac{n}{d^2}\sum_{\ell\mid d}\mu\left(\frac{d}{\ell}\right)\frac{p^\ell-1}{p-1}.$$

\medskip
{\bf (II)}~~$\mathrm{Irr}(G/P_3| P/P_3)$

\smallskip
The character theory of the group $P/P_3$ is studied in~\cite{HO} and \cite{11}. Observe that $P/P_3$ has nilpotency class $2$ and that $n$ is odd. These two facts, together with $\gcd(n, p(p^n -1)) = 1$,  imply that the conditions described in~\cite[page~340]{HO} or in~\cite[page~190]{11} are satisfied.

We write $\overline{G} = G/P_3$ and use the bar convention.
From~\cite[Theorem~4.4]{11}, we see that $\overline{P}$ has
\begin{itemize}
\item $1$ principal character,
\item $p^n-1$ non-principal characters of degree $1$,
\item $p(p^n-1)$ characters of degree $p^{(n-1)/2}$.
\end{itemize}

By Lemma~\ref{iso},  $\irr{\overline{P}}$  and   $\cl{\overline{P}}$  are isomorphic $\overline{C}\overline{H}$-sets.
By~\cite[Theorem 4.1]{HO}, $$\{ (1+a_1 x + a_2 x^2)^c \mid a_1, a_2 \in F_p, c \in \overline{C} \}$$ is  a set of
representatives for the conjugacy classes of $\overline{P}$.
Hence,  $\cent{\overline{P}}{\overline{H}} = \{ 1+a_1 x + a_2 x^2 \mid a_1, a_2 \in F_p\}$
has non-trivial intersection
with every $\overline{C}$-orbit in  ${\rm Cl}(\overline{P})$.
Let $\theta \in \irr{\overline{P}}$ with $\theta \neq 1_{\overline{P}}$.
By the above mentioned  isomorphism of $\overline{C}\overline{H}$-sets, it follows that  there
exists a $c \in \overline{C}$ such that $\theta^c$ is $\overline{H}$-invariant.
Moreover, $\overline{C}$ acts semi-regularly on $\irr{\overline P} \setminus \{1_{\overline P} \}$,
hence  $I_{\overline{G}}(\theta^c) =\overline{P}\overline{H}$.
So, by Gallagher's theorem and Clifford correspondence, $\irr{\overline{G}|\theta^c}$ consists of $n$ irreducible characters of degree $\theta(1)|C|$.

Moreover, the $\overline{G}$-orbits
on $\irr{\overline{P}}$ coincide with the $\overline{C}$-orbits and,
by the semi-regular action of $\overline{C}$,
in $\irr{\overline{P}}$ there are $p-1$ $\overline{G}$-orbits of linear characters and $p(p-1)$ $\overline{G}$-orbits of characters of degree $p^{(n-1)/2}$.
Therefore, $\mathrm{Irr}(G/P_3| P/P_3)$ consists of
\begin{itemize}
\item $n(p-1)$ characters of degree $(p^n-1)/(p-1)$,
\item $np(p-1)$ characters of degree $p^{(n-1)/2}(p^n-1)/(p-1)$.
\end{itemize}

So far, using~\eqref{condition}, we have shown that the character degree graph of $G/P_3$ has two connected components, which are complete graphs. One connected component consists of prime divisors of $n$ and the other connected component consists of the prime divisors of $p(p^n-1)/(p-1)$.

\medskip
{\bf (III)}~~$\mathrm{Irr}(G| P_3)$

\smallskip
Consistently with the notation in Lemma~\ref{lce}, we let $X = \widehat{P_3} \setminus \{ 1_{P_3}\} $ and let $X_0$ be the union of the $H$-invariant $C$-orbits (that is, the $G$-orbits) in $X$. Recall that the $H$-invariant $C$-orbits in $X$ are the $G$-orbits in $X$.
Hence,   $$\mathrm{Irr}(G| P_3) = \mathrm{Irr}(G|X) =\mathrm{Irr}(G|X \setminus X_0) \cup \mathrm{Irr}(G|X_0) .$$

\medskip
{\bf (III.A)}~~$\mathrm{Irr}(G|X \setminus X_0)$

\smallskip
Let $\varphi\in X \setminus X_0$ (i.e.  $\varphi$ does not lie in an $H$-invariant $C$-orbit).
By Lemma~\ref{lemma1} and  part (a) of Lemma~\ref{lce}, we can assume that $I_G(\varphi) = PDS$.
By Lemma~\ref{lemma:tech2} part~(b), for every $\mu\in\mathrm{Irr}(B)$,   $I_P(\mu\times \varphi) = P_2$.
So, the induced character $\theta= (\mu\times \varphi)^P$ is irreducible and has degree $p^{n}$.
Hence, by  \cite[Problem 6.3]{isaacs} $\varphi$ is fully ramified with respect to $P/P_3$  and
$\irr{P|\varphi} = \{ \theta \}$. It follows that $I_G(\theta) = I_G(\varphi) = PDS$.
Therefore, since $\theta$ extends to $PDS$ by~\cite[Corollary 6.28]{isaacs} and $DS$ is abelian, by Gallagher's theorem  $\irr{G|\varphi} = \irr{G|\theta}$ consists of $|DS| = n(p^2+p+1)/3$ characters of degree $p^n|ER| =3\cdot p^{n}\cdot (p^n-1)/(p^3-1)$.

By Lemma~\ref{lce} parts~(a) and~(b), $X \setminus X_0$ contains exactly  $(|X| - (p-1)|E|)/|ER| = (p^3-p)/3$ $G$-orbits, so we conclude that
$\irr{G|X\setminus X_0}$ contains exactly $\frac{n(p^3-p)(p^2+p+1)}{9}$ characters, all having degree $3\cdot p^{n}\cdot (p^n-1)/(p^3-1)$.

\medskip
{\bf (III.B)}~~$\irr{G|X_0}$

\smallskip
Let $\varphi_0\in X_0$ be an  $H$-invariant character.
By part (b) of Lemma~\ref{lce} $\{ \varphi_0^i \mid i = 1, \ldots, p-1\}$
is a set of representatives  for the $G$-orbits in $X_0$, so
\begin{equation}\label{r0}
\irr{G|X_0} = \bigcup_{i = 1}^{p-1} \irr{G|\varphi_0^i}.
\end{equation}
Observe that this is a disjoint union by Lemma~\ref{lemma1}.

We also remark that $\{ \varphi_0^i \mid i = 1, \ldots, p-1\}$ is an orbit under the action of the Galois group
$\mathrm{Gal}(\mathbb{Q}(e^{2\pi i/p})/\mathbb{Q})$, so for every $2 \leq i \leq p-1$ there is a degree-preserving bijection
between $\irr{G|\varphi_0^i}$ and $\irr{G|\varphi_0}$.
Moreover, the sets $\irr{G|\varphi_0^i}$ are pairwise disjoint, for $i = 1, \ldots, p-1$, since the characters
$\varphi_0^i$ lie in distinct $G$-orbits.

Let $Y = \irr{P_2|\varphi_0}$
be the set consisting of  the $p^n$ extensions of $\varphi_0$ to  the abelian group  $P_2$. Since $PDH=I_G(\varphi_0)$, we deduce that $Y$ is a $PDH$-set.
We observe that,  if $\psi_1 = \psi^g$ for $\psi, \psi_1 \in Y$ and $g \in G$, then
$g \in I_G(\varphi_0) = PDH$. Therefore $\psi,\psi_1\in Y$ are $G$-conjugate if and only if they are $PDH$-conjugate.
Hence, for every orbit $Z$ of $G$ on $\irr {P_2}$, either $Z\cap Y=\emptyset$ or $Z\cap Y$ is  a $PDH$-orbit. 

Let  $Y_0$ be the $PDH$-orbit of $\psi_0 = 1_B \times \varphi_0$ in $Y$ and let $\mathcal{Y}$ be a set of representatives for the $PDH$-orbits in $Y \setminus Y_0$.
Then $|\mathcal{Y}| = p-1$ by part (e) of Lemma~\ref{lemmaY}. Moreover, by the previous paragraph and Lemma~\ref{lemma1}, $\irr{G|\varphi_0}$ can be expressed as
$$\irr{G|\varphi_0} = \irr{G|Y} =  \irr{G|Y \setminus Y_0} \cup \irr{G|Y_0} = \bigcup_{\psi \in \mathcal{Y}} \irr{G|\psi}  \cup  \irr{G|Y_0},$$
where all unions are disjoint.


 \medskip
{\bf (III.B.1)}~~$\irr{G|Y \setminus Y_0}$

 \smallskip
 Let  $\psi \in \mathcal{Y}$.
 By Lemma~\ref{lemma1}
 and part (b) of Lemma~\ref{lemmaY} we can assume that $\psi$ is $H$-invariant.

 Since $\psi = \mu \times \varphi_0 \not\in Y_0$,
 then by part (a) of Lemma~\ref{lemma:tech2} $\mu \neq (\varphi_0)_{a}$ for every $a \in F$  and hence
by part (a) of Lemma~\ref{lemmaY} $B \cap Q$ is not contained in $\mathrm{ker}(\mu)$.

Let $I=I_{P}(\psi)$. We claim  that $\psi$ does not extend to $I$.\footnote{We denote with $U'$ the derived subgroup of a finite group $U$.} In fact, if the linear character $\psi$ extends to $I$, then
$I'\cap B \leq \mathrm{ker}(\mu \times \varphi_0) \cap B = \mathrm{ker}(\mu)$.
But, as $I = P_2Q$, $I' = P_3Q'$ and, using the notation $\overline{P} = P/P_3$, $\overline{I'} = \overline{Q'} = \overline{Q}'$.
By \cite[Corollary 4.4]{isaacs1}, $\overline{Q}' = \overline{Q} \cap \overline{B} = \overline{Q \cap B}$ and hence
it follows $Q \cap B \leq I'\cap B \leq \mathrm{ker}(\mu)$, a contradiction.

By Lemma~\ref{lemma2}, there exists a subgroup $U$,  with  $P_2 \leq U \leq I$,   such
that all characters $\tau \in \irr{U|\psi}$ are extensions of  $\psi$ and are fully ramified with respect to $I/U$.
Since $\psi$ does not extend to $I$, then $U < I$ and hence $|I:U| = p^2$.
As $\psi$ is $H$-invariant and $\gcd(|H|, |U|) = 1$, by \cite[Corollary 13.30]{isaacs} at least one of the characters in $\irr{U|\psi}$ is $H$-invariant. But,  since $\gcd(p-1, |H|) = 1$, $H$ acts trivially on $U/P_2$, and hence on its dual group.
Therefore, by Gallagher's theorem all characters in $\irr{U|\psi}$ are $H$-invariant and, since they are fully ramified with respect to
$I/U$,  the same is true for all the characters  in $\irr{I|\psi}$.
Hence,  $H \leq I_G(\xi^P)$ for every $\xi \in \irr{I|\psi}$.

We claim that $I_{CH}(\xi^P) = H$ for all $\xi \in \irr{I|\psi}$.
In fact, $I_E(\xi^P) = 1$ by part (b) of Lemma~\ref{lemma:frobenius},  and if $y \in I_D(\xi^P)$, then by Clifford's theorem  $y$ fixes the $P$-orbit of $\psi$,
so $y = 1$ by part (d) of Lemma~\ref{lemmaY}. Thus,  $I_{C}(\xi^P) = 1$ and the claim follows.

Recall that $|\irr{I|\psi}| = |\irr{U|\psi}| = p$ and hence $\irr{I|\psi} = \{ \xi_1, \ldots, \xi_p\}$. So, by Clifford's correspondence, $\irr {P|\psi}=\{\xi_1^P,\ldots,\xi_p^P\}$. As for every $i\in \{1,\ldots,p\}$, $\xi_i^P$ extends to $I_{G}(\xi_i^P)=PH$, by Gallagher's theorem  we get  that
$\irr{G| \xi_i^P}$ contains $n$ characters, all of degree $\xi_i^P(1)|C| = p^{n-2}(p^n-1)/(p-1)$.

If $\irr{G|\xi_i^P} \cap \irr{G|\xi_j^P} \neq \emptyset$, then $\xi_i^P$ and $\xi_j^P$ are $G$-conjugate characters.
As $\varphi_0$ is the only irreducible constituent of  the restriction of both $\xi_i^P$ and $\xi_j^P$ to $P_3$,
then $\xi_i^P$ and $\xi_j^P$ are in the same orbit under the action
of $I_G(\varphi_0) = PDH$.  But both characters are $PH$-invariant, so there exists an  element $y \in D$ such that $(\xi_i^P)^y = \xi_j^P$.
Since  $(\xi_i^P)^y $ lies over the $P$-orbit of $\psi^y$ in $Y$, while $\xi_j^P$ lies over the $P$-orbit of $\psi$ in $Y$,
by part (d) of Lemma~\ref{lemmaY} we conclude that $y = 1$, and hence $i = j$.
Therefore, for every $\psi \in \mathcal{Y}$, $\irr{G|\psi} = \bigcup_{i=1}^p \irr{G|\xi_i^P}$ contains $n \cdot p$ irreducible characters.

 As $|\mathcal{Y}| = p-1$,
$$\irr{G| Y \setminus Y_0} = \bigcup_{\psi \in \mathcal{Y}} \irr{G|\psi}$$
consists of
$n\cdot p(p-1)$ characters of degree $p^{n-2}(p^n-1)/(p-1)$.

\medskip
{\bf (III.B.2)}~~$\irr{G|Y_0}$

\smallskip
Recall $\psi_0=1_B\times \varphi_0$. Let $I=I_P(\psi_0)= P_2Q$ (by Lemma~\ref{lemma:tech2} part~(b)) and let $L = I/P_2 \simeq Q/(Q \cap P_2)$;
 as usual, we identify $L$ and $F_{p^3}$ by the isomorphism $\omega_1$  defined by, for $a \in F_{p^3}$,
$(1+a x) +P_2  \mapsto a$.
Under this identification, $D$ acts transitively, and hence regularly,  on the $1$-dimensional subspaces of the
$3$-dimensional $F_p$-space $L$.  Hence, $L$ is an irreducible $D$-module.
 Then by  Lemma~\ref{lemma2} and by the fact that $|L|$ is not a square
 (so $\psi_0$ cannot be  fully ramified with respect to $L$), it follows that $\psi_0$ extends to $I$.
 By \cite[Theorem 13.28]{isaacs}, there is a $D$-invariant extension $\xi_0$ of $\psi_0$ to $I$ and, by Gallagher's theorem
 $\irr{I| \psi_0} = \{ \xi_0 \tau \mid \tau \in \widehat{L} \}$.
 Again,  $\widehat{L}$ and $L$ are  isomorphic $D$-sets
 and, as $D$ acts fixed-point-freely on $L$, we deduce that $D$ acts fixed-point-freely on $\widehat{L}$ and hence $D$ acts
 semi-regularly on $\irr{I| \psi_0} \setminus \{ \xi_0 \}$.
 In particular, $\xi_0$ is the only $D$-invariant character in $\irr{I| \psi_0}$.

 Moreover, as $D$ acts regularly on the $1$-dimensional subspaces of $L$ and $\cent{L}R$ is one of them,
 it follows that every $D$-orbit in $L$,
 and hence in $\widehat{L}$, contains an $R$-invariant element.
 Therefore, for every $\xi = \xi_0\tau \in \irr{I| \psi_0}$, there exists an element $y \in D$ such that $\xi^y = (\xi_0\tau)^y = \xi_0\tau^y$ is
 $R$-invariant. Since $\psi_0 = 1_B \times \varphi_0$ is $S$-invariant,  a similar argument shows that all the  characters in $\irr{I| \psi_0}$ are $S$-invariant.
 Hence, there exists a set $\mathcal{Z}$ of representatives of the $D$-orbits in $\irr{I|\psi_0}$ such that every
 $\xi \in \mathcal{Z}$ is $H$-invariant. Thus, $\mathcal{Z}$ is in fact a set of representatives for the
 $DH$-orbits in $\irr{I|\psi_0}$. We observe that necessarily $\xi_0 \in \mathcal{Z}$ and that
 $|\mathcal{Z}\setminus \{ \xi_0 \}| = (|L| -1)/|D| = p-1$, as $D$ acts semi-regularly on $\irr{I| \psi_0} \setminus \{ \xi_0\}$.

 Let $W = \irr{P| Y_0}  =  \irr{P| \psi_0}$.
 If two characters $\theta, \theta_1 \in W$ lie in the same $G$-orbit, then (being $P$-invariant) there exists an element
 $g \in CH$ such that $\theta_1 = \theta^g$.
 Since $\varphi_0$ is the only irreducible constituent of  both $\theta_{P_3}$ and $(\theta_1)_{P_3}$, it follows that
  $g \in I_{CH}(\varphi_0) = DH$.

By Clifford correspondence, the mapping $\xi \mapsto \xi^P$ is a bijection between $\irr{I| \psi_0}$ and $W$.
Therefore,  $\mathcal{Z}^P = \{\xi^P \mid \xi \in \mathcal{Z}\}$ is a set of representatives for the  $DH$-orbits in $W$. By the previous paragraph and Lemma~\ref{lemma1}, we hence have that
$$\irr{G|Y_0}=\irr{G|W}=\bigcup_{\xi\in \mathcal{Z}}\irr{G|\xi^P}$$
is a disjoint union.

We claim  that, for every $\xi \in \mathcal{Z}$ with $\xi \neq \xi_0$,   $I_{CH}(\xi^P) = H$.
 In fact, if $y \in I_D(\xi^P)$, then (by \cite[Theorem 6.11(c)]{isaacs}, observing that both $I$ and
 $\psi_0$ are $D$-invariant) $y$ fixes the unique  irreducible constituent of $(\xi^P)_I$ that lies over $\psi_0$, that is $\xi$ itself.
 Since $D$ acts semi-regularly on $\irr{I| \psi_0}\setminus\{\xi_0\}$ and $\xi \neq \xi_0$, it follows $y= 1$. So,  $I_D(\xi^P) = 1$ and,
 recalling part (b) of Lemma~\ref{lemma:frobenius}, we deduce $I_C(\xi^P) = 1$.
 On the other hand, for $y \in H$, $(\xi^P)^y = (\xi^y)^P = \xi^P$, so  $\xi^P$ is $H$-invariant, and the claim is proved.
 We also observe that, since $\xi_0$ is $D$-invariant,  the same argument proves that $I_{CH}(\xi_0^P) = DH$.

 Hence, for each $\theta \in \mathcal{Z}^P$ with $\theta \neq \xi_0^P$,  by Gallagher's theorem we have that $\irr{G|\theta}$ contains
 $n$ characters of degree $p^{n-3}|C|$. In fact, $\theta$ extends to $I_G(\theta)=PH$, because  $H$ is cyclic.
Therefore,  $\irr{G|\mathcal{Z}^P \setminus\{\xi_0^P\}}$ contains $n(p-1)$ characters, each of  degree $p^{n-3}(p^n -1)/(p-1)$.

 On the other hand,  $\xi_0^P$ extends to  $I_G(\xi_0^P) = PDH$, but  $PDH/P \simeq DH = S \times DR$, where $DR$ is a Frobenius group with complement $R$.
 So, $DH$ has $n$ linear characters and $ |S| (|D|-1)/|R| = \frac{n}{9}(p^2+p)$ irreducible characters of degree $|R| = 3$.
Hence, by Gallagher's theorem, $\irr{G|\xi_0^P}$ contains $n$ irreducible characters
 of degree $\xi_0^P(1)|E| = p^{n-3}(p^n-1)/(p^3 -1)$ and $\frac{n}{9}(p^2+p)$ irreducible characters of degree $3p^{n-3}(p^n-1)/(p^3 -1)$.

\medskip We summarize here the character degrees of $G$ and their multiplicities.
\begin{description}
\item[(I)] For each divisor $d$ of $n$, $\mathrm{Irr}(G/P)$ has $$\frac{n}{d^2}\sum_{\ell\mid d}\mu\left(\frac{d}{\ell}\right)\frac{p^\ell-1}{p-1}$$ characters of degree $d$,
\item[(II)] $\mathrm{Irr}(G/P_3| P/P_3)$ has $n(p-1)$ characters of degree $(p^n-1)/(p-1)$ and $np(p-1)$ characters of degree $p^{\frac{n-1}{2}}(p^n-1)/(p-1)$,
\item[(III)] $\mathrm{Irr}(G| P_3)$ has
  \begin{description}
\item[(III.A)] $n(p^3-p)(p^2+p+1)/9$ characters of degree $3 p^n (p^n-1)/(p^3-1)$,

  \item[(III.B)] and, recalling  (\ref{r0}) we have
 \begin{description}
\item[(III.B.1)]$np(p-1)^2$ characters of degree $p^{n-2} (p^n-1)/(p-1)$,
\item[(III.B.2)]
\begin{itemize}
\item $n(p-1)^2$ characters of degree $p^{n-3}(p^n-1)/(p-1)$,
\item  $n(p-1)$ characters of degree $p^{n-3} (p^n-1)/(p^3-1)$, and
\item $n(p-1)(p^2+p)/9$ characters of degree $3 p^{n-3}(p^n-1)/(p^3-1)$.
\end{itemize}
  \end{description}
\end{description}
\end{description}

Hence,
we see that  the character degree graph $\Delta = \Delta(G)$ is a graph of diameter $3$, with
$\alpha_{\Delta}  = \pi(n/3)$, $\beta_{\Delta} = \{ 3\}$, $\gamma_{\Delta} = \pi((p^n-1)/(p^3-1))$ and
$\gamma_{\Delta} = \pi(p^2+p+1)$, concluding the proof.
\end{proof}

We now shift our focus to proving Theorem~B. Before proceeding, we establish some fundamental notation.

Let $G$ be an arbitrary solvable group. We denote the first two terms of the Fitting series of $G$ as $\mathrm{F}(G)$ and $\mathrm{F}_2(G)$, respectively. Additionally, $\mathrm{Z}(G)$ denotes the center of $G$.

Consider a prime $p$ and a positive integer $n$. Let $F = F_{p^n}$ be a field with $p^n$ elements and let $F^\times = F \setminus \{ 0 \}$.
Let $\Gamma(p^n)$
be the group of semilinear transformations of $F$  and let $\Gamma_0(p^n)$ be  the  subgroup $F^\times$ within $\Gamma(p^n)$.
 By definition, we have:
$$\Gamma(p^n)=\Gamma_0(p^n)\rtimes\mathrm{Gal}(F|F_p)\cong C_{p^n-1}\rtimes C_n.$$

The following result, part of which is Theorem~B, utilizes the notation introduced in Section~\ref{sec:intro} (for $\pi_0,\pi_1,\alpha_\Delta,\beta_\Delta,\gamma_\Delta,\delta_\Delta$). Recall that, for a natural number $n$ and a set of primes $\sigma$, we denote by $n_\sigma$ the largest divisor $k$ of $n$ such that $\pi(k) \subseteq \sigma$.

\begin{theorem}\label{primesets}
  Let $G$ be a solvable group  and assume that $\Delta=\Delta(G)$ has diameter three. Let $\pi_0,\pi_1,\alpha = \alpha_\Delta,\beta_\Delta,\gamma_\Delta,\delta_\Delta$ be as in Section~$\ref{sec:intro}$.
  Then there exist a prime $p$ and  a positive integer $n$  such that, setting  $n_0 = n_{\pi_0}$, we have
  \begin{description}
  \item[(a)] $n_0$ is odd,
  \item[(b)]  $\{ p \} \cup \pi\left(\frac{p^n -1}{p^{n/n_0}-1}\right) \subseteq \pi_1 \subseteq \{ p \} \cup \pi(p^n-1)$,
  \item[(c)]    $\{ p \} \cup \pi\left(\frac{p^n -1}{p^{n/n_{\alpha}}-1}\right) \subseteq \gamma_{\Delta}$,

\item[(d)] the subgraph of $\Delta$ induced on $\beta_{\Delta} \cup \{ p \} \cup \pi\left(\frac{p^n -1}{p^{n/n_{\alpha}}-1}\right)$
is a clique,
\item[(e)]$|\gamma_{\Delta}| \geq 2^{|\beta_{\Delta}|}\left( 2^{|\alpha|} -1  \right) + 1 .$
\end{description}
\end{theorem}

\begin{proof}
  Let $G$ be a solvable group such that $\Delta = \Delta(G)$ has diameter three.
  Then,  by~\cite[Theorem~A]{CDPS} $G = PH$, where $P$ is a non-abelian Sylow $p$-subgroup of $G$, 
  $H$ is a $p$-complement of $G$ and  $\mathrm{F}(G) = P \times A$, where $A = \cent HP \leq \mathrm{Z}(G)$.
Let  $P_1 = [P, G]$ and $P_i = \gamma_i(P)$ for $i\in\{2,\ldots, c\}$, where  $c$ is the nilpotency class of $P$ and $\gamma_i(P)$ is the $i^{\mathrm{th}}$ term of the lower central series of $P$.

By~\cite[Theorem~A]{CDPS},   all factors $M_i = P_i/P_{i+1}$ for  $i \in \{1, \ldots,  c\}$ are chief factors of $G$ of the same
  order, say $p^n$, and each group $G/{\bf C}_G(M_i)$ acts irreducibly on the module  $M_i$ as a  subgroup of the  semilinear
  group $\Gamma(p^n)$.

Moreover, \cite[Theorem~A]{CDPS} and  \cite[Remark 4.4]{CDPS} tell us that $\Delta(G/P_3)$ is a  disconnected graph,  with the same vertex set as $\Delta(G)$, and that the connected components of
  $\Delta(G/P_3)$ are 
\begin{align*}\pi_0 &= \pi(|G/\mathrm{F}_2(G)|) = \pi(|H/\mathrm{F}(H)|), \hbox{and}\\
\pi_1& = \{ p \} \cup \pi(|\mathrm{F}_2(G)/\mathrm{F}(G)|) = \{ p \} \cup \pi(|\mathrm{F}(H)/A|).
\end{align*} 

By \cite[Theorem~A(b)]{CDPS} all the Sylow subgroups of $H/A$ are cyclic and hence by ~\cite[Corollary 11.22]{isaacs} and
\cite[Corollary 11.31]{isaacs} each irreducible character of $\mathrm{F}(G)$ extends to its inertia subgroup in $G$.
Thus, $\Delta(G) = \Delta(G/A)$ by~\cite[Lemma 2.4]{CDPS} and hence, by replacing $G$ with $G/A$, we can assume that $A  = 1$ and $P = \mathrm{F}(G)$.

As $[P, H] \leq P_1$, by coprimality $\cent H{M_1}$ acts trivially on $P/P_2$ and hence (as $P_2 = P'$)  on the quotient of $P$ over its Frattini subgroup. Since
$|H|$ is coprime to $p$, it follows that $\cent H{M_1} = \cent HP = 1$. In particular, $H$ acts faithfully by conjugation on $M_1$.

Let  $m = |\mathrm{F}(H)|$ and  $n_0 = |H/\mathrm{F}(H)|$; so 
$\pi_0 = \pi(n_0)$, in particular $n_0\ne1$, and $\pi_1 = \{p \} \cup \pi(m)$.
Seeing  $H$, in its action on on $M_1$,   as a  subgroup of $\Gamma(p^n)$, let $H_0 = H \cap \Gamma_0(p^n)$. So $H_0$
is a cyclic normal subgroup of $H$ and hence $|H_0|$ divides $m$. 
As $H_0$ acts semi-regularly on $\irr{M_1}$, Clifford theory implies that, for every $\mu \in \irr{M_1}$,
$I_H(\mu)$ contains a Hall $\pi(n_0)$-subgroup of $H$. Hence, \cite[Lemma~3.5]{CDPS} yields that
$$m_0 = \frac{p^n -1}{p^{n/n_0}-1}$$
divides
$|H_0|$ and hence $m_0$ divides $m$, proving the first inclusion in part~(b). 

We now show that $n_0$ is odd, getting part~(a). In fact,  if $2 \in \pi(n_0) = \pi_0$, then $p \neq 2$, because $p$ belongs to $\pi_1$ which
has empty intersection with $\pi_0$. Hence, $p \equiv 1 \pmod 2$ implies that   $m_0  \equiv n_0 \equiv 0 \pmod 2$, so
$m$ is even and hence $\pi_0\cap \pi_1 \neq \emptyset$, a contradiction. 

We now observe that there exists  a primitive prime divisor of  $p^n -1$. Otherwise, since $n_0$ is an odd divisor of $n$, by Zsigmondy's theorem
$p=2$ and $n = 6$; so $n_0= 3$ and $m_0 =21$, giving again the contradiction $\pi_0\cap \pi_1 \neq \emptyset$.
As a consequence, by~\cite[Lemma 3.7]{CDPS} $\mathrm{F}(H) = H_0$, so $m$ divides $p^n -1$ and hence 
 $\pi_1 \subseteq \{ p \} \cup \pi(p^n-1)$, finishing the proof of  part~(b).

Write $\alpha = \alpha_{\Delta}$ and $\beta = \beta_{\Delta}$; 
so $\pi_0  = \alpha \cup \beta$ and  $n_0 = n_{\alpha} n_{\beta}$. Let $q = p^{n/n_0}$ (so $m_0 = (q^{n_0}-1)/(q-1)$)   and define
$$m_1 = \frac{q^{n_0}-1}{q^{n_0/n_{\alpha}}-1}  = \frac{q^{n_0}-1}{q^{n_{\beta}}-1}.$$

Let now $t \in \beta$. We will show that $t$ is adjacent in $\Delta(G)$ to every prime in $\{p\}\cup\pi(m_1)$.
By definition, there exists a prime $s \in  \gamma_{\Delta}$ and a character
$\chi\in \irr G $ such that $ts$ divides $\chi(1)$.
In particular, setting $K = \ker(\chi) \cap P$,  $\Delta(G/K)$ is connected, so $P_3 \not\leq K$ and
$\Delta(G/K) = \Delta(G)$. Hence, 
replacing $G$ with $G/K$, we can assume  $K = 1$.

Let $M = P_c $;  so $M\leq \mathrm{Z}(P)$ and $M$ is a faithful irreducible  $\bar H$-module, where  $\bar H= H/{\bf C}_H(M)$.
Seeing $\bar H$ as a subgroup of $\Gamma(p^n)$, we set $\bar H_0=\bar H\cap \Gamma_0(p^n)$.
For every $\mu \in \irr M$ and $\psi \in \irr{P|\mu}$, $\psi_M = \psi(1)\mu$  is a homogeneous character, and  hence
$I_H(\psi) \leq I_H(\mu)$.
Since $p$ divides the degree of each character in $\irr{G|M}$ and $p$ is not adjacent in $\Delta(G)$ to any vertex belonging
to $\alpha$, 
Clifford theory implies that, for every $\psi \in \irr{P|M}$, $I_H(\psi)$ contains a
Hall $\alpha$-subgroup of $H$. Therefore, for every non-principal $\mu \in \irr M$, $I_{\bar H}(\mu)$ contains a
Hall $\alpha$-subgroup of $\bar H$ and  hence, by~\cite[Lemma~3.5]{CDPS},   $m_1$ divides $|\bar H_0|$.   
Since $\bar H_0$ acts semi-regularly $\irr M$, we  deduce that
\begin{equation}
  \label{eq:d}
\{t\}\cup \{ p\} \cup \pi(m_1)\subseteq \pi(\chi(1)).  
\end{equation}

In particular, $\{p\}\cup\pi(m_1) \subseteq \gamma_{\Delta}$, proving part~(c).

Moreover, since (\ref{eq:d}) holds for every  $t\in \beta$ in $\Delta(G)$, we conclude that
$\beta \cup \{p\}\cup\pi(m_1)$ induces a clique in $\Delta(G)$,  which shows part~(d).

Recall that, denoting by $d(k)$ the number of positive divisors of an integer $k$, if $k = p_1^{e_1}p_2^{e_2}\cdots p_v^{e_v}$, with
$p_1,p_2, \ldots, p_v$ distinct primes and $e_1, e_2,\ldots, e_v$ positive integers, then
$$d(k) = (e_1+1)(e_2+1)\cdots(e_v+1).$$  
Writing $n_0 = \prod_{r \in \alpha}r^{e_r} \cdot \prod_{t \in \beta}t^{e_t}$ with $e_r, e_t \geq 1$ for all $r\in \alpha$, $t\in \beta$, 
 Lemma~\ref{lemma:nt}   implies that
 $$|\pi(m_1)| \geq d(n_0)  - d(n_{\beta}) = \prod_{t \in \beta}(e_t +1) \left( \prod_{r \in \alpha}(e_r+1) -1 \right) \geq
 2^{|\beta|}(2^{|\alpha|} -1).$$
Hence,  part~(e) now follows from part~(c).
\end{proof}

We finish by mentioning that ~\cite[Question, page~91]{HHHI} asks for the existence of a solvable group $G$ such that $\Delta=\Delta(G)$ has diameter three, with  $|\alpha_{\Delta}| = 2$,   $|\beta_\Delta|=  1$
(so, writing $\beta = \{ q \}$, $q$ is a cut-vertex of $\Delta$) and $|\V\Delta|=11$. Notice that $11$ is the smallest cardinality
for a putative $\V\Delta$, with the given conditions on $\alpha$ and $\beta$. 
Indeed, if such a graph $\Delta$ exists, then Theorem~\ref{primesets} shows that  $|\gamma_\Delta|\ge 2^{|\beta_\Delta|}(2^{\alpha_\Delta}-1)+1\ge 2\cdot 3+1=7$ and  $q$ is adjacent to all other vertices of $\Delta$, except one. 
Using the notation and the arguments in the proof of Theorem~\ref{primesets}, one can also show that  if $\Delta = \Delta(G)$ has eleven vertices,
$\beta = \{q\}$ and $\alpha = \{t_1, t_2\}$, then $n = n_0 = q t_1t_2$ is a product of three distinct odd primes.
Moreover, by the observation following Lemma~\ref{lemma:nt}, we deduce that $p= 2$  and $\lambda_d(2)$ is a prime power for every divisor $d > 1$
of $n$. Ultimately, our inability to determine the existence of $\Delta$ leaves us unable to conclusively address Conjecture~\ref{conj:1}. The presence of such a graph would refute it.

Nevertheless, we can prove the existence of a similar graph with twelve vertices. 
Let  $p= 2$ and  $n = 3 \cdot 5 \cdot 13$. A computation shows that $\gcd(2^n-1,n)=1$, so~\eqref{condition} is satisfied and hence we may apply the results from Section~\ref{construction} and from the proof of Theorem~A. We obtain a character degree graph $\Delta$ having diameter three, where
$\alpha_\Delta = \{5, 13\}$, $\beta_\Delta= \{3\}$, $|\gamma_\Delta| = 8$ and  $\delta_\Delta=\{7\}$. In particular, $|\V\Delta|=12$.

\end{document}